\documentclass[a4paper,12pt]{article}

\usepackage{amssymb,amsmath,amsthm,mathrsfs,latexsym,times}

\usepackage[authoryear,round]{natbib}

\usepackage[pdftex]{hyperref}
\hypersetup{plainpages=True, pdfstartview=FitV,
colorlinks=true,linkcolor=blue,citecolor=blue}

\def\dtv{d_{\mathrm{TV}}}
\def\dd#1{{\,\mathrm{d}}{#1}}

\newtheorem{thm}{Theorem}[section]
\newtheorem{cor}[thm]{Corollary}
\newtheorem{lem}[thm]{Lemma}
\newtheorem{prop}[thm]{Proposition}
\theoremstyle{remark}
\newtheorem{rem}[thm]{Remark}

\title{The estimate of $\chi^2$ distance between binomial and generalized binomial distributions}
\author{Vytas Zacharovas\\
	Institute of Computer Science\\
	Vilnius University\\
	Naugarduko 24, Vilnius, Lithuania\\
	E-mail: vytas.zacharovas@mif.vu.lt}

\begin{document}
\maketitle

\begin{abstract}
We obtain the estimate of difference between binomial and generalized binomial distributions in \(\chi^2\) metric and in several other related metrics. 
\end{abstract}

\section{Introduction}
We will investigate the distribution of a sum 
\[
S_n=I_1+I_2+\cdots+I_n
\] 
of $n$ independent
indicators $I_j$ taking value $1$ with probability
$p_j=\mathbb{P}(I_j=1)$ and \(0\) with probability \(1-p_j\). We will refer to the above distribution as \emph{generalized binomial distribution}. The  case when all $p_j$ are equal $p_1=p_2=\cdots =p_n=p$ corresponds to the case of  \emph{simple binomial distribution}  $\mathscr{B}(n,p)$ taking value \(j\) with probability
\[
b(n,p;j)=\binom{n}{j}p^jq^{n-j}
\] 
for all \(0\leqslant j \leqslant n\), where $q=1-p$.
 The generalized binomial distribution 
has the mean value
\[
\mathbb{E}S_n=p_1+p_2+\cdots+p_n.
\]
Thus if we chose
\begin{equation}
	\label{mean_of_p_j}
	p=\frac{1}{n}\sum_{j=1}^{n}p_j
\end{equation}
then the distribution of $S_n$  will have the same mean value as the simple binomial distribution $\mathscr{B}(n,p)$. In what follows we will denote \(q=1-p\).

 Thus it is  natural to try approximate the distribution of $S_n$ by the distribution of $\mathscr{B}(n,p)$ where \(p\) equal to the arithmetical average (\ref{mean_of_p_j}) of \(p_j\). 
 
 In what follows we assume that not all \(p_j\) are identical and equal to either \(0\) or \(1\) that is we will not consider the case when \(p_1=p_2=\cdots=p_n\in\{0,1\}\). This assumption implies that \(0<p<1\) and \(0<q<1\). Let us denote
 \begin{equation}
 \label{define_delta}
 \delta_m:=\frac{1}{n(pq)^{m/2}}\sum_{j=1}^{n}|p_j-p|^m 
 \end{equation}
In what follows for simplicity sake we will denote \(\delta=\delta_2\). 
 
  \cite{Ehm19917} investigated the difference between the distributions of $S_n$ and $\mathscr{B}(n,p)$ 
in total variation distance 
\[
\dtv(\mathscr{L}(S_n),\mathscr{B}(n,p)):=\frac{1}{2}\sum_{j=0}^{n}\left|\mathbb{P}(S_n=j)-b(n,p;j)\right|
\]
and proved the inequality
\begin{equation}
\label{ehms_bound}
C(1-p^{n+1}-q^{n+1})\delta\leqslant\dtv(\mathscr{L}(S_n),\mathscr{B}(n,p))\leqslant (1-p^{n+1}-q^{n+1})\frac{n}{(n+1)}\delta
\end{equation}
where  \(C>0\) is an absolute constant and \(\delta=\delta_2\).

This result was further improved by \cite{roos_2000} who obtained asymptotic expansion
of the difference of generalized binomial distribution \(\mathscr{L}(S_n)\) and \(\mathscr{B}(n,p)\).
The main result of our paper will be the estimate of  the $\chi^2$-distance between $S_n$ and $\mathscr{B}(n,p)$ defined as
\[
    \chi^2(\mathscr{L}(S_n),\mathscr{B}(n,p))
    =\sum_{j=0}^n\left|\frac{\mathbb{P}(S_n=j)}
    {b(n,p,j)}-1\right|^2b(n,p,j).
\]
This quantity is correctly defined whenever \(0<p<1\) which is always satisfied if \(S_n\) is not equal to constant with probability \(1\).
We prove the following theorem.
\begin{thm}\label{main_thm} For all $n\geqslant 2$ and $\delta<1$ hold the inequalities
\[
\frac{ n}{2(n-1)}\delta^2\leqslant\chi^2(\mathscr{L}(S_n),\mathscr{B}(n,p))
\leqslant \frac{n}{2(n-1)}\delta^2\left(1+O\left(\frac{\delta}{1-\delta}+\frac{1}{\sqrt{n}}\frac{\delta_3}{\delta}\right)\right).
\]
\end{thm}
Note that the constants in the \(O(\ldots)\) symbol can be made explicit (see the inequality of Proposition \ref{P_upper_bound}) however since the resulting expression is somewhat cumbersome we have chosen to suppress the exact constants in the main formulation of our result.

The condition requiring $\delta$ to be smaller than $1$ is not very restrictive since as was noted in \cite{Ehm19917} this quantity can be expressed as
\[
\delta=1-\frac{\mathbb{V}S_n}{npq}
\]
and as a consequence $\delta$ never exceeds \(1\)  with equality $\delta=1$ being possible only when  $S_n$ is a constant with probability \(1\).

Thus our estimates imply that 
\[
\frac{\chi^2(\mathscr{L}(S_n),\mathscr{B}(n,p))}{\delta^2}\sim \frac{ n}{2(n-1)}
\]
if \(\delta\to 0\) and  \(\delta_3/(\sqrt{n}\delta) \to 0\) . Note that the term
\(\delta_3/(\sqrt{n}\delta)\) does not exceed \(\sqrt{5}\) (see Lemma \ref{L_ineq_delta}), which means that \(\chi^2(\mathscr{L}(S_n),\mathscr{B}(n,p))=O(\delta^2)\) if \(\delta\) does not exceed some fixed constant smaller than \(1\).

Note that although the mean values of \(S_n\) and \(\mathscr{B}(n,p)\) coincide, their corresponding variances \((1-\delta)npq\) and \(npq\)  can differ considerably if \(\delta\) is not small enough. A number of papers were  devoted to approximating generalized binomial distribution \(\mathscr{L}(S_n)\) by simple binomial distribution \(\mathscr{B}(n^*,p^*)\) where the parameters \(p^*\) and \(n^*\) are chosen in such a way as to minimize the difference of both mean and variances of \(S_n\) and \(\mathscr{B}(n^*,p^*)\). It is shown that considerable improvement of closeness of approximation is obtained in this way (see e.g. \cite{pekoz_rollin_cekanavicius_shwartz_2009} and references therein). We expect that the approach we develop in this paper can also be applied to this setting also.

Unlike Ehm's approach that is based on Stein's method, the main idea of our proof is purely analytic and relies on the integral form of  Parseval identity for the Krawtchouk polynomials. The proof follows the same pattern that was first used in \cite{zacharovas_hwang_2010} to evaluate the $\chi^2$ distance between Poisson distribution and generalized Bernoulli distribution. 
\subsection{The estimates for other probability distances}

The estimate for the \(\chi^2\) distance can be used to get upper bounds for a number of other probability metrics. For example, by trivial application of Cauchy inequality
we immediately get upper  bound for the total variation distance
 \[
     \dtv(\mathscr{L}(S_n),\mathscr{B}(n,p))
     \leqslant  \frac{1}{2}\sqrt{\chi^2(\mathscr{L}(S_n),\mathscr{B}(n,p))}.
 \]
  thus replacing here \(\chi^2(\mathscr{L}(S_n),\mathscr{B}(n,p))\) by its upper bound provided by  Theorem \ref{main_thm} we obtain the estimate
   \[
       \dtv(\mathscr{L}(S_n),\mathscr{B}(n,p))
       \leqslant  \frac{1}{2^{3/2}}\delta \sqrt{\frac{n}{n-1}\left(1+O\left(\frac{\delta}{1-\delta}+\frac{1}{\sqrt{n}}\frac{\delta_3}{\delta}\right)\right)}.
   \]
   Since $1/2^{3/2}=0.353553\ldots$ the above bound can be smaller than Ehm's upper bound (\ref{ehms_bound})
  for sufficiently small $\delta$ and sufficiently large $n$ as $p$ is fixed $0<p<1$. The constant $1/2^{3/2}=0.353553$ of the above inequality is not optimal since as was shown by \cite{roos_2000} in his Theorem 3, the optimal upper bound contains constant \(1/\sqrt{2\pi e}=0.2419707\) in its leading term.

The upper bound for $\chi^2$ also provides the upper bound for 
   Kullback-Leibner divergence (or information divergence) defined as
   \[
   d_{\text{KL}}(\mathscr{L}(S_n),\mathscr{B}(n,p))
   :=\sum_{j=0}^n \mathbb{P}(S_n=j)\log
   \frac{\mathbb{P}(S_n=j)}
   {b(n,p,j)}
   \]
   due to the simple inequality
   \[
   d_{\text{KL}}(\mathscr{L}(S_n),\mathscr{B}(n,p))  \leqslant\chi^2(\mathscr{L}(S_n),\mathscr{B}(n,p)).
   \]
In a similar fashion the upper bound for $\chi^2$ quickly leads to a non-uniform bound for the  difference of distribution functions. Indeed, suppose \(K_n\) is a random variable distributed as a simple binomial variable \(\mathscr{B}(n,p)\). Then application of Cauchy inequality gives the estimate
    \[
    \bigl|\mathbb{P}(S_n\leqslant x)-\mathbb{P}(K_n\leqslant x)\bigr|\leqslant \sqrt{\chi^2(\mathscr{L}(S_n),\mathscr{B}(n,p))}\sqrt{R(x)}
    \]
    where
    \[
    R(x):=\min \bigl\{\mathbb{P}(K_n\leqslant x),\mathbb{P}(K_n>x) \bigr\}.
    \]
    Clearly $R(x)\leqslant 1/2$ therefore the above estimate after taking the supremum over all \(x\in \mathbb{R}\) leads to the upper bound for the Kolomogorov's distance.

\section{Proofs}
\subsection{ Krawtchouk - Parseval identity}
In order to investigate the \(\chi^2\) metric we will need a formula expressing the weighted sum of squares of numbers \(a_0,a_1,\ldots,a_n\) in terms of the generating function of coefficients these numbers. Such expression was obtained in \cite{chen_hwang_zacharovas_2014} as a consequence of the orthogonality property of Krawtchouk polynomials and the related Parseval identity. In view of importance of this identity for further analysis we provide here its new proof that is purely analytic and does not involve Krawtchouk polynomials. In fact the identity of the following theorem can serve as a starting point for deriving Krawtchouk polynomials as the polynomial that are orthogonal with respect to binomial measure.The following proof is of independent interest and can be generalized and applied to other classes of orthogonal polynomials.
\begin{thm}[Krawtchouk - Parseval identity, \cite{chen_hwang_zacharovas_2014}]
\label{krawtchouk-parseval_integral}
Suppose
\[
F(z)
=\sum_{k=0}^na_k z^k,
\]
then 
\[
 \sum_{k=0}^n\frac{|a_k|^2}{\binom{n}{k}p^kq^{n-k}}
 =(n+1)\int_0^\infty\frac{J_n\left(F,p;\sqrt{\frac{u}{pq}}\right)}{(1+u)^{n+2}}\dd u,
\]
where
\begin{equation}
    J_n(F,p;r)
    :=\frac{1}{2\pi}\int_{-\pi}^\pi
    \left|(1-pre^{it})^n
    F\left(\frac{1+qre^{it}}{1-pre^{it}}\right)\right|^2\dd t.
\end{equation}
\end{thm}
\begin{proof}

By Parseval identity
\[
    \frac{1}{2\pi}\int_{-\pi}^{\pi}|F(re^{it})|^2\,dt
    =\sum_{k=0}^n|a_k|^2 r^{2k}.
\]
Replacing here $r=\sqrt{uq/p}$, multiplying both sides of this equation by $1/(1+u)^{n+2}$ and integrating by $u$ from $0$ to $+\infty$ we obtain
\[
    \int_{0}^{\infty}\frac{1}{(1+u)^{n+2}}\left(\frac{1}{2\pi}\int_{-\pi}^{\pi}\bigl|F(\sqrt{uq/p}e^{it})\bigr|^2\,dt\right)\,du
    =\sum_{k=0}^n|a_k|^2(q/p)^k\int_{0}^{\infty}\frac{u^{k}}{(1+u)^{n+2}}\,du .
\]
Introducing a change of variables $u\to u^2$ into the integral on the left side of the above identity and noting that
\begin{equation}
\label{int_u}
(n+1)\int_0^\infty\frac{u^k\,du}{(1+u)^{n+2}}
=\binom{n}{k}^{-1}
\end{equation}
 we get
\[
   q^{-n}(n+1)\frac{1}{\pi} \int_{0}^{\infty}\frac{u}{(1+u^2)^{n+2}}\left(\int_{-\pi}^{\pi}\bigl|F(ue^{it}\sqrt{q/p})\bigr|^2\,dt\right)\,du
    = \sum_{k=0}^n\frac{|a_k|^2}{\binom{n}{k}p^kq^{n-k}}.
\]
Note that the double integral can be regarded as being
obtained from an integral over all complex plane
\[
q^{-n}(n+1)  \frac{1}{\pi} \int_{\mathbb{C}}\frac{\bigl|F(z\sqrt{q/p})\bigr|^2}{(1+|z|^2)^{n+2}}\,dx\,dy
\]
where \(z=x+iy\) by passing to polar coordinates $z=x+iy=re^{it}$. Thus
\[
  q^{-n}(n+1)  \frac{1}{\pi} \int_{\mathbb{C}}\frac{\bigl|F(z\sqrt{q/p})\bigr|^2}{(1+|z|^2)^{n+2}}\,dx\,dy= \sum_{k=0}^n\frac{|a_k|^2}{\binom{n}{k}p^kq^{n-k}}.
\]
let us make a new change of variables
\[
z=\sqrt{p/q}\frac{1+qw}{1-pw}
\]
 taking into account that the determinant of the Jacobian matrix of such transform is $(p/q)/|1-pw|^4$ we get
\[
\begin{split}
 \sum_{k=0}^n\frac{|a_k|^2}{\binom{n}{k}p^kq^{n-k}}&= q^{-n}(n+1) \frac{1}{\pi}  \int_{\mathbb{C}}\frac{\left|F\left(\frac{1+qw}{1-pw}\right)\right|^2}{\left(1+(p/q)\left|\frac{1+qw}{1-pw}\right|^2\right)^{n+2}}\frac{p/q}{|1-pw|^4}\,|dw|
 \\
 &= pq(n+1)  \frac{1}{\pi} \int_{\mathbb{C}}\frac{\left|F\left(\frac{1+qw}{1-pw}\right)\right|^2|1-pw|^{2n}}{\left(q\left|{1-pw}\right|^2+q\left|{1+qw}\right|^2\right)^{n+2}}\,|dw|
 \\
  &= pq(n+1)  \frac{1}{\pi} \int_{\mathbb{C}}\frac{\left|F\left(\frac{1+qw}{1-pw}\right)\right|^2|1-pw|^{2n}}{\left(1+pq|w|^2\right)^{n+2}}\,|dw|
\end{split}
\]
Here we used the fact that
\[
q\left|{1-pw}\right|^2+p\left|{1+qw}\right|^2=1+pq|w|^2
\]
Introducing now a change to polar coordinates $w=re^{it}$ we get
\[
\begin{split}
 &\sum_{k=0}^n\frac{|a_k|^2}{\binom{n}{k}p^kq^{n-k}}\\
 &= 2pq(n+1)   \int_{0}^\infty\frac{r}{{\left(1+pq|r|^2\right)^{n+2}}}\left(\frac{1}{2\pi}\int_{-\pi}^\pi
 \left|(1-pre^{it})^n
 F\left(\frac{1+qre^{it}}{1-pre^{it}}\right)\right|^2\dd t\right)\,dr
\end{split}
\]
Noting that the internal integral coincides with  $J_n(F,p;r)$ as defined in the formulation of the theorem, we can rewrite our identity as
\[
 \sum_{k=0}^n\frac{|a_k|^2}{\binom{n}{k}p^kq^{n-k}}= 2pq(n+1)   \int_{0}^\infty\frac{r}{{\left(1+pq|r|^2\right)^{n+2}}}J_n(F,p;r)\,dr
\]
which after the change of variables $r=\sqrt{u/(pq)}$ takes the form of the identity stated in the formulation of the theorem.

\end{proof}
\begin{cor}
	\label{cor_parseval}
	Let $F(z)$ be a polynomial 
	\[
	F(z)
	=\sum_{k=0}^na_k z^k,
	\] then
\[
    \sum_{k=0}^n\frac{|a_k|^2}{\binom{n}{k}p^kq^{n-k}}=\sum_{j=0}^n\frac{|c_j|^2}{\binom{n}{j}(pq)^j},
\]
where $c_0,c_1,\ldots$ are the Taylor coefficients in the expansion
\[
  (1-pw)^nF\left(\frac{1+qw}{1-pw}\right)
    =\sum_{j=0}^n c_jw^j.
\]
\end{cor}
\begin{proof}
By Parseval identity
\[
\begin{split}
    J_n(F,p;r)
    =\frac{1}{2\pi}\int_{-\pi}^\pi
    \Bigl|\sum_{j=0}^n c_jw^j\Bigr |^2\dd t=\sum_{j=0}^n|c_j|^2 r^{2j}.
\end{split}
\]
Plugging this expression of $J_n(F,p;r)$ into the integral inside the identity of Theorem \ref{krawtchouk-parseval_integral} and using the expression for the integral (\ref{int_u}) we obtain the proof of the Corollary.
\end{proof}

\subsection{The generalized binomial distribution}
Let us apply Corollary \ref{cor_parseval} with \(a_j=\mathbb{P}(S_n=j)\). The generating function of  coefficients \(a_j\) will be equal to
\[
f(z)=\sum_{j=0}^{n}\mathbb{P}(S_n=j)z^j=\prod_{j=1}^{n}(q_j+p_jz)
\]
where $q_j=1-p_j$. Then  Corollary \ref{cor_parseval} leads to identity
\begin{equation}
\label{main_identity_1}
\sum_{ t=0}^n\left|\frac{\mathbb{P}(S_n=t)}{b(n,p,t)}\right|^2b(n,p,t)=\sum_{j=0}^n\frac{|c_j|^2}{\binom{n}{j}(pq)^j}
\end{equation}
where \(c_j\) are the coefficient of the polymomial
\[
(1-pw)^n f\left(\frac{1+qw}{1-pw}\right)=\sum_{j=0}^{n}c_jw^j
\]
Which after a few simple calculations on the left side of the above equation yields the identity
\[
\prod_{j=1}^{n}\bigl(1+(p_j-p)w\bigr)=\sum_{j=0}^{n}c_jw^j
\]
Hence computing the first and second derivatives of the above expression and recalling the definition (\ref{define_delta}) of $\delta$ we obtain
\[
c_0=1,\quad c_1=0,\quad c_2=-\frac{1}{2}\sum_{j=1}^{n}(p_j-p)^2=-\frac{1}{2}pqn\delta 
\]
Discarding all but the first three terms of the right side of identity (\ref{main_identity_1}) we obtain the lower bound for the sum
\begin{equation}
\label{lowerebound}
\sum_{t=0}^n\left|\frac{\mathbb{P}(S_n=t)}{b(n,p,t)}\right|^2b(n,p,t)\geqslant |c_0|^2+\frac{|c_1|^2}{\binom{n}{1}pq}+\frac{|c_2|^2}{\binom{n}{2}(pq)^2}=1+\frac{ n}{2(n-1)}\delta^2.
\end{equation}
When \(n=2\) the above inequality turns into identity. 

If \(n=3\) then \(c_3=(p_1-p)(p_2-p)(p_3-p)\) and thus  by Cauchy inequality stating that geometric average does not exceed the arithmetic mean we get
\[
\sqrt[3]{ |p_1-p|^3|p_2-p|^3|p_3-p|^3}\leqslant \frac{|p_1-p|^3+|p_2-p|^3+|p_3-p|^3}{3}
\]
or in our notations \(|c_3|\leqslant (pq)^{3/2}\delta_3\). Hence we immediately obtain the inequality
\begin{equation}
\label{lowerebound_1}
\begin{split}
\sum_{t=0}^3\left|\frac{\mathbb{P}(S_3=t)}{b(3,p,t)}\right|^2b(3,p,t)&= |c_0|^2+\frac{|c_1|^2}{\binom{3}{1}pq}+\frac{|c_2|^2}{\binom{3}{2}(pq)^2}
+\frac{|c_3|^2}{\binom{3}{3}(pq)^3}
\\
&\leqslant 1+\frac{ 3}{4}\delta^2+\delta_3^2
\end{split}
\end{equation}

The inequality of the following Lemma is proved inside Lemma 2.2 of \cite{roos_2014}, however in view of its importance to our argument we provide its proof here.
\begin{lem}
\label{inequality_1}
Suppose $x_1,x_2,\ldots,x_n\in \mathbb{R}$ are such that 
\[
x_1+x_2+\cdots+x_n=0
\]
then the inequality 
\[
\left|\prod_{j=1}^{n}(1+x_jz) \right|^{2}\leqslant
\left(1+|z|^2\frac{1}{n}\sum_{j=1}^{n}x_j^2\right)^n
\]
holds for all complex $z\in \mathbb{C}$.
\end{lem}
\begin{proof}
Applying Cauchy inequality we obtain
\[
\left|\prod_{j=1}^{n}(1+x_jz) \right|^{2}\leqslant
\left(\frac{1}{n}\sum_{j=1}^{n}|1+x_jz|^2\right)^n
\]
Note that for any complex $w$ we have $|1+w|^2=1+2\Re w+|w|^2$. This gives us
\[
\begin{split}
\left|\prod_{j=1}^{n}(1+x_jz) \right|^{2}&\leqslant
\left(\frac{1}{n}\sum_{j=1}^{n}(1+2x_j\Re z+x_j^2|z|^2)\right)^n
\\
&\leqslant \left(1+|z|^2\frac{1}{n}\sum_{j=1}^{n}x_j^2\right)^n
\end{split}
\]
since by condition of the theorem $x_1+x_2+\cdots+x_n=0$.
\end{proof}

\begin{thm}
	\label{upper_bound_simple} Whenever  \(0<\delta<1\) holds the inequality
\[
\sum_{t=0}^n\left|\frac{\mathbb{P}(S_n=t)}{b(n,p,t)}\right|^2b(n,p,t)
\leqslant\frac{1-\delta^{n+1}}{1-\delta}.
\]
\end{thm}
\begin{proof}
Applying the identity of Theorem \ref{krawtchouk-parseval_integral} with \(a_j=\mathbb{P}(S_n=j)\) we can express the sum on the left hand side of the identity in the formulation of the theorem as
\begin{equation}
\label{eq_a_j=P}
\sum_{t=0}^n\left|\frac{\mathbb{P}(S_n=t)}{b(n,p,t)}\right|^2b(n,p,t)
=(n+1)\int_0^\infty\frac{J_n\left(f,p;\sqrt{\frac{u}{pq}}\right)}
{(1+u)^{n+2}}\dd u
\end{equation}
where
\[
\begin{split}
J_n(f,p;r)
=\frac{1}{2\pi}\int_{-\pi}^\pi
\left|\prod_{j=1}^{n}\bigl(1+(p_j-p)re^{it}\bigr)\right|^2\dd t 
\end{split}
\]
Applying inequality of Lemma \ref{inequality_1} with \(x_j=p_j-p\) and \(z=re^{it}\)  
we can evaluate the above the above integral as
\[
J_n(f,p;r)
     \leqslant \left(1+r^2pq\delta\right)^n
\]
Using this inequality to evaluate the integral on right hand side of the identity (\ref{eq_a_j=P}) we obtain 
\[
\begin{split}
\sum_{t=0}^n\left|\frac{\mathbb{P}(S_n=t)}{b(n,p,t)}\right|^2b(n,p,t)
    \leqslant(n+1)\int_0^\infty\frac{\left(1+u\delta\right)^n}{(1+u)^{n+2}}\dd u,
    \end{split}
\]
Introducing change of variables $u=1/y-1$ in the  integral on the right hand side of the above inequality and obtain an explicit expression for it
\[
(n+1)\int_0^\infty\frac{\left(1+u\delta\right)^n}{(1+u)^{n+2}}\dd u
        =(n+1)\int_0^1\left(y+(1-y)\delta\right)^n\dd u
        =\frac{1-\delta^{n+1}}{1-\delta}
\]

\end{proof}
\begin{cor}\label{chi<=n} For all \(n\geqslant 2\) and \(\delta<1\) holds the inequality  
\[
\frac{ n}{2(n-1)}\delta^2\leqslant\chi^2
(\mathscr{L}(S_n),\mathscr{B}(n,p))\leqslant
\delta\frac{1-\delta^{n}}{1-\delta}.
\]
If \(n=2\) or \(n=3\) we have more accurate estimates
\[
\chi^2(\mathscr{L}(S_2),\mathscr{B}(2,p))=\delta^2
\]
and
\[
\frac{3}{4}\delta^2\leqslant\chi^2(\mathscr{L}(S_3),\mathscr{B}(3,p))
\leqslant \frac{3}{4}\delta^2+\delta_3^2.
\]
\end{cor}
\begin{proof}
Let us note that
\[\chi^2
(\mathscr{L}(S_n),\mathscr{B}(n,p))=\sum_{ t=0}^n\left|\frac{\mathbb{P}(S_n=t)}{b(n,p,t)}\right|^2b(n,p,t)-1
\]
hence the lower bound for the \(\chi^2\) will follow from  inequality (\ref{lowerebound}) for the sum in the above identity
\[
\chi^2
(\mathscr{L}(S_n),\mathscr{B}(n,p))\geqslant\frac{ n}{2(n-1)}\delta^2
\]
while the the inequality  of the Theorem \ref{upper_bound_simple} provides the upper bound 
\[
\chi^2
(\mathscr{L}(S_n),\mathscr{B}(n,p))
\leqslant
\frac{1-\delta^{n+1}}{1-\delta}-1
=
\delta\frac{1-\delta^{n}}{1-\delta}
\]
\end{proof}
The upper bound for \(\chi^2\) of Theorem \ref{upper_bound_simple} is \(O(\delta)\)  as \(\delta\to 0\) and thus is far from optimal for small \(\delta\). In order to show that upper bound can be improved to \(O(\delta^2)\) we will need  more refined versions of the inequality for product of complex numbers than the one provided by Lemma \ref{inequality_1}.

\begin{lem}[\cite{zacharovas_hwang_2010}]
	\label{inequality_3} For any complex numbers $\{v_k\}$, the following
	inequality holds
	\begin{equation}\label{ineq-P1}
	\begin{split}
	\left| \prod_{1\le k\le n}(1+v_k)e^{-v_k}
	-1\right|
	\leqslant \frac{V_2}{2}+\left(\frac{c_1}{4} V_2^2
	+c_2V_3 \right) e^{V_2/2},
	\end{split}
	\end{equation}
	where 
	\[
	V_2 := \sum_{1\le k\le n} |v_k|^2 \quad \hbox{and}\quad V_3 := \sum_{1\le k\le n} |v_k|^3
	\]

	$c_1 = \sqrt{e}-1\approx 0.6487$ and 
	\[
	c_2 = \frac12\int_0^1 e^{t^2/2}(1-t^2) dt
	\approx 0.3706.
	\]
\end{lem}

\begin{prop}
	\label{P_upper_bound} For all $n\geqslant 4$ and $\delta<1$ holds the inequality
\[
\begin{split}
\chi^2(\mathscr{L}(S_n),\mathscr{B}(n,p))
&\leqslant\frac{n}{2 (n - 1) }\Biggl(\delta+\frac{a_1n}{\sqrt{(n-3)(n-2)}  }\delta^2
+\frac{a_2}{ \sqrt{n - 2} }\delta_3\Biggr)^2
\\
&\quad+\frac{n}{n-1}\frac{6e}{1-\delta}\delta^3
\end{split}
\]
where \(a_1=\sqrt{3}c_1e^{1/2}\approx 1.8525\) and \(a_2=2\sqrt{3}c_2e^{1/2}\approx2.1166\) and \(c_1\), \(c_2\) are the same constants as in the formulation of Lemma \ref{inequality_3}.
	
\end{prop}
\begin{proof}
Applying the Krawtchouk - Parseval  identity of Theorem \ref{krawtchouk-parseval_integral} with \(F(z)=g(z)\)  where
\[
g(z)=\prod_{j=1}^{n}(q_j+p_jz)-(pz+q)^n=\sum_{k=0}^{n}\bigl(\mathbb{P}(S_n=k)-b(n,p,k)\bigr)z^k
\]
we can express the \(\chi^2\) metric as
\begin{equation}
\label{main_int}
\chi^2(\mathscr{L}(S_n),\mathscr{B}(n,p))
    =(n+1)\int_0^\infty\frac{J_n\left(g,p;\sqrt{\frac{u}{pq}}\right)}
    {(1+u)^{n+2}}\dd u
\end{equation}
where
\[
\begin{split}
J_n(g,p;r)
    &=\frac{1}{2\pi}\int_{-\pi}^\pi\left|\prod_{j=1}^{n}\bigl(1+(p_j-p)re^{it}\bigr)-1\right|^2 \dd t
    \end{split}
\]
We will evaluate the integral on the right side of  identity (\ref{main_int}) by first  obtaining upper bounds for the function under the integration sign by applying the estimates provided by Lemma \ref{inequality_1} and Lemma \ref{inequality_3}.

 In order to make the expression of \(J_n(g,p;r)\) ready for the application of Lemma \ref{inequality_3} let us note first that 
\[
\prod_{j=1}^{n}e^{-(p_j-p)re^{it}}=1
\]
therefore we can rewrite 
\[
\begin{split}
J_n(g,p;r)
&=\frac{1}{2\pi}\int_{-\pi}^\pi\left|\prod_{j=1}^{n}\bigl(1+(p_j-p)re^{it}\bigr)e^{-(p_j-p)re^{it}}-1\right|^2 \dd t.
\end{split}
\]
Applying the inequality of Lemma \ref{inequality_3}  with \(v_j=(p_j-p)re^{it}\) we obtain an upper bound for the quantity under the integration sign of the above integral in terms of quantities
\[
V_2= r^2npq\delta_2\quad\hbox{and} \quad V_3= r^3n(pq)^{3/2}\delta_3
\]
which leads to the estimate
\begin{equation}
\label{inequality_j_n}
J_n\left(g,p;\sqrt{\frac{u}{pq}}\right)\leqslant \left(\frac{un\delta}{2}+\left(\frac{c_1}{4} (un\delta)^2
+c_2u^{3/2}n\delta_3 \right) e^{un\delta/2}\right)^2
\end{equation}
using the above inequality we can estimate
\[
\begin{split}
&\sqrt{\int_0^{1/(n\delta)}\frac{J_n\left(g,p;\sqrt{\frac{u}{pq}}\right)}
	{(1+u)^{n+2}}\dd u}
\\
&\leqslant \sqrt{\int_0^{1/(n\delta)}\left(\frac{un\delta}{2}+\left(\frac{c_1}{4} (un\delta)^2
+c_2u^{3/2}n\delta_3 \right) e^{1/2}\right)^2\frac{\dd u}
{(1+u)^{n+2}}}
\end{split}
\]
Applying Minkowski inequality to the above integral we  estimate  it by a sum of three integrals
\[
\begin{split}
&\sqrt{\int_0^{1/(n\delta)}\frac{J_n\left(g,p;\sqrt{\frac{u}{pq}}\right)}
	{(1+u)^{n+2}}\dd u}
\\
&\leqslant
\frac{n\delta}{2}\sqrt{\int_0^{\infty}\frac{u^2\dd u}
	{(1+u)^{n+2}}}+\frac{c_1(n\delta)^2e^{1/2}}{4}\sqrt{\int_0^{\infty}\frac{u^4\dd u}
	{(1+u)^{n+2}}}
\\
&\quad\quad\quad\quad\quad\quad+c_2n\delta_3e^{1/2}\sqrt{\int_0^{\infty}\frac{u^3\dd u}
	{(1+u)^{n+2}}}
\end{split}
\]
Note that the integrals inside the left side of the above inequality can be expressed in terms of binomial coefficients according to our previously encountered formula (\ref{int_u}). As a result we get
\[
\begin{split}
&\int_0^{1/(n\delta)}\frac{J_n\left(g,p;\sqrt{\frac{u}{pq}}\right)}
{(1+u)^{n+2}}\dd u
\\
&\leqslant
\left(\frac{n\delta}{2}\sqrt{\frac{\binom{n}{2}^{-1}}{n + 1}}+\frac{c_1(n\delta)^2e^{1/2}}{4}\sqrt{\frac{\binom{n}{4}^{-1}}{n + 1}}
+c_2n\delta_3e^{1/2}\sqrt{\frac{\binom{n}{3}^{-1}}{ n + 1}}\right)^2
\end{split}
\]
Further simplifying the above estimate we get
\begin{equation}
\label{int_part_1}
\begin{split}
&(n+1)\int_0^{1/(n\delta)}\frac{J_n\left(g,p;\sqrt{\frac{u}{pq}}\right)}
{(1+u)^{n+2}}\dd u
\\
&\leqslant
\frac{n}{2 (n - 1) }\Biggl(\delta+c_1\delta^2e^{1/2}\sqrt{\frac{3n^2}{(n-3)(n-2)  }}
+c_2\delta_3e^{1/2}\sqrt{\frac{12}{ n - 2 }}\Biggr)^2
\end{split}
\end{equation}
The estimate (\ref{inequality_j_n})  contains a rapidly increasing multiplier \(e^{un\delta/2}\) and as such would result in a divergent integral if applied to evaluate the integral on the right hand side of the Krawtchouk - Parseval identity (\ref{main_int}) for large \(u\). For evaluating \(J_n(g,p;r)\) for large \(r\) we note that
\[
J_n(g,p;r)
=\frac{1}{2\pi}\int_{-\pi}^\pi\left|\prod_{j=1}^{n}\bigl(1+(p_j-p)re^{it}\bigr)\right|^2 \dd t-1
\]
and applying inequality of Lemma \ref{inequality_1} with \(z=re^{it}\) and \(x_j=p_j-p\) to evaluate the product under the integration sign we obtain the estimate
\begin{equation}
\label{J_inequality_1}
J_n(g,p;r) \leqslant \left(1+r^2pq\delta\right)^n-1
\end{equation}
whose upper bound is a  polynomial of degree \(2n\). Hence we can estimate the remaining integral
\[
(n+1)\int_{1/(n\delta)}^\infty\frac{J_n\left(g,p;\sqrt{\frac{u}{pq}}\right)}
{(1+u)^{n+2}}\dd u<(n+1)\int_{1/(n\delta)}^\infty\frac{(1+u\delta)^n}
{(1+u)^{n+2}}\dd u
\]
Making a change of variables $u=1/y-1$ in the last integral  we get
\[
\begin{split}
(n+1)\int_{1/(n\delta)}^\infty\frac{(1+u\delta)^n}
        {(1+u)^{n+2}}\dd u&=(n+1)\int_{0}^{1/\left(1+\frac{1}{n\delta}\right)}\bigl(y+\delta(1-y)\bigr)^n\dd y
        \\
        &=\frac{1}{1-\delta}\left(\left(\frac{1+\frac{1}{n}}{1+\frac{1}{n\delta}}\right)^{n+1}-\delta^{n+1}\right)
\\
 &\leqslant\frac{1}{1-\delta}\left(\frac{1+\frac{1}{n}}{1+\frac{1}{n\delta}}\right)^{n+1}
 \\
  &\leqslant\frac{e\left(1+\frac{1}{n}\right)}{1-\delta}\frac{6n^2\delta^3}{n^2-1}
       \end{split}
\]
here in the last step we used the inequality
\[
\left(1+\frac{1}{n\delta}\right)^{n+1}\geqslant C_{n+1}^3\frac{1}{n^3\delta^3}=\frac{n^2-1}{6n^2\delta^3}
\]
Hence finally we obtain the estimate
\begin{equation}
\label{int_part_2}
(n+1)\int_{1/(n\delta)}^\infty\frac{J_n\left(g,p;\sqrt{\frac{u}{pq}}\right)}
        {(1+u)^{n+2}}\dd u\leqslant \frac{n}{n-1}\frac{6e}{1-\delta}\delta^3.
\end{equation}

Splitting the integral on the right side of Krawtchouk - Parseval identity identity (\ref{main_int}) into two the integrals over intervals \((0,1/(n\delta))\) and \((1/(n\delta),\infty)\) and applying the obtained upper bounds (\ref{int_part_1}) and (\ref{int_part_2}) for these integrals   we complete the proof of the proposition.

\end{proof}
\begin{rem}
 As follows from the proof of Proposition \ref{P_upper_bound} a more precise estimate   can be obtained by replacing the term \(\frac{n}{n-1}\frac{6e}{1-\delta}\delta^3\) in the main inequality of Proposition \ref{P_upper_bound} with a smaller but more complicated term \(\frac{1}{1-\delta}\left(\left(\frac{1+\frac{1}{n}}{1+\frac{1}{n\delta}}\right)^{n+1}-\delta^{n+1}\right)\).
\end{rem}

\begin{proof}[Proof of Theorem \ref{main_thm}] The lower bound of the  inequality  follows from Corollary \ref{chi<=n}. The upper bound follows from the estimate of Proposition \ref{P_upper_bound} for \(n\geqslant 4\). If \(n\leqslant3\) the required estimates follows from the estimates (\ref{lowerebound}) and (\ref{lowerebound_1}) we obtained earlier by considering the cases \(n=2\) and \(n=3\).
\end{proof}
The following Lemma shows that the error terms inside \(O(\ldots)\)  of our main result presented in Theorem \ref{main_thm} are bounded when \(\delta\) does not exceed some constant smaller than \(1\).
\begin{lem} 
	\label{L_ineq_delta}
	For all \(n\geqslant 2\) holds the inequality
\[
\frac{1}{\sqrt{n}}\frac{\delta_3}{\delta}\leqslant \sqrt{2+\sqrt{\frac{8}{n}}+\frac{1}{n}}.
\]
\end{lem}
\begin{proof}

Let us at first consider the case when \(p\leqslant 1/2\). In such case \(q=1-p\geqslant 1/2\) and as a consequence
\[
\begin{split}
\frac{1}{n}\left(\frac{\delta_3}{\delta}\right)^2&=\frac{1}{npq}\left(\frac{\sum_{j=1}^n|p-p_j|^3}{\sum_{j=1}^n(p-p_j)^2}\right)^2
\\
&\leqslant  \frac{2}{np}\left(\frac{\sum_{j=1}^np(p-p_j)^2+\sum_{j=1}^np_j(p-p_j)^2}{\sum_{j=1}^n(p-p_j)^2}\right)^2
\\
&\leqslant  \frac{2}{np}\left(p+\frac{\sum_{j=1}^np_j(p-p_j)^2}{\sum_{j=1}^n(p-p_j)^2}\right)^2
\\
&\leqslant \frac{2}{np}\left(p^2+2p\frac{\sum_{j=1}^np_j(p-p_j)^2}{\sum_{j=1}^n(p-p_j)^2}+ \left(\frac{\sum_{j=1}^np_j(p-p_j)^2}{\sum_{j=1}^n(p-p_j)^2}\right)^2\right)
\end{split}
\]
Applying the Cauchy inequality to estimate the numerator of the last fraction we get
\[
\begin{split}
\left(\frac{\sum_{j=1}^np_j(p-p_j)^2}{\sum_{j=1}^n(p-p_j)^2}\right)^2
&\leqslant \frac{\sum_{j=1}^np_j^2\sum_{j=1}^n(p-p_j)^4}{\left(\sum_{j=1}^n(p-p_j)^2\right)^2}
\\
&\leqslant  np\frac{\sum_{j=1}^n(p-p_j)^4}{\sum_{j=1}^n(p-p_j)^4+\sum_{\substack{1\leqslant i,j \leqslant n\\
	i\not=j}}(p-p_i)^2(p-p_j)^2}
\\
&\leqslant np
\end{split}
\]
here we have taken into account that \(0\leqslant p_j\leqslant 1\) which implies that
\[
\sum_{j=1}^np_j^2\leqslant \sum_{j=1}^np_j=np.
\]
Hence
\[
\frac{1}{n}\left(\frac{\delta_3}{\delta}\right)^2\leqslant \frac{2}{np}(p^2+2p\sqrt{np}+np)=\frac{2}{n}(p+2\sqrt{np}+n)\leqslant
\frac{2}{n}\left(\frac{1}{2}+2\sqrt{n/2}+n\right)\
\]
If \(p>1/2\) then we can repeat the same argument with \(q_j=1-p_j\) replacing \(p-p_j=q_j-q\) where \(q\) is an arithmetic average of \(q_j\).
\end{proof}

 \section*{Acknowledgments} The author sincerely thanks Prof. Hsien-Kuei Hwang for   discussions on the topic of the paper as well as for his hospitality during the author's visits to Academia Sinica (Taiwan).

\bibliographystyle{apalike}

\end{document}